\documentclass[10pt]{amsart}						

	\usepackage[pagebackref,  pdfpagemode=FullScreen,  colorlinks=true]{hyperref}	   

	\usepackage{verbatim}
	\usepackage{mathrsfs, amssymb}

	\newtheorem{thm}{Theorem}
  	
  	\newtheorem{lem}{Lemma}
  	\newtheorem{prop}{Proposition}

	\theoremstyle{definition}

	\theoremstyle{remark}
  	\newtheorem*{rem}{Remark}
  	
	\newtheorem*{ex}{Example}

	\newcommand{\M}{\mathcal{M}}
	\newcommand{\Mbar}{\overline{\mathcal{M}}}

	\makeatletter
	\@namedef{subjclassname@2010}{%
	\textup{2010} Mathematics Subject Classification}
	\makeatother

\title{Divisors of secant planes to curves}

\author{Nicola Tarasca}
\address{University of Utah, Department of Mathematics, 155 S 1400 E, Salt Lake City, UT 84112, USA}
\email{tarasca@math.utah.edu}


\begin{document}

\begin{abstract}
Inside the symmetric product of a very general curve, we consider the codimension-one subvarieties of symmetric tuples of points imposing exceptional secant conditions
on linear series on the curve of fixed degree and dimension. We compute the classes of such divisors, and thus obtain improved bounds for the slope
of the cone of effective divisor classes on symmetric products of a very general curve.
By letting the moduli of the curve vary, we study more generally the classes of the related divisors inside the moduli space of stable pointed curves.
\end{abstract}

\maketitle

Let $C$ be a smooth curve of genus $g$. The {\it $n$-symmetric product} $C_n$ of $C$ is the smooth variety that parametrizes divisors of $C$ of degree $n$. A geometric way to obtain subvarieties of $C_n$ is to consider {\it secant divisors} to linear series of $C$. Recall that a {\it linear series} of type $\mathfrak{g}^r_d$ is a pair $\ell=(L,V)$ of a line bundle $L\in {\rm Pic}^d(C)$ and a linear subspace $V\subset H^0(C,L)$ of dimension $r+1$. The Brill-Noether theorem says that the variety $G^r_d(C)$ of linear series of type $\mathfrak{g}^r_d$ on a general curve $C$ of genus $g$ has dimension $\rho(g,r,d):=g-(r+1)(g-d+r)$, and is empty if $\rho(g,r,d)<0$.

Given a general curve $C$ of genus $g$, the variety of divisors of $C$ imposing exceptional secant conditions on a {\it fixed} linear series of $C$ has been classically studied (\cite[\S VIII.4]{MR770932}). In the following, we consider the locus in $C_n$ of degree-$n$ divisors of $C$  imposing at most $t$ independent conditions on some {\it arbitrary} linear series in $G^r_d(C)$, that is,
\[
\mathfrak{S}^{r,t}_{g,d} (C) := \left\{ D\in C_n : \exists \, \ell \in G^r_{d}(C) \,\,\mbox{with}\,\, h^0\left( C, \ell\otimes \mathcal{O}_C (-D) \right) \geq r+1-t   \right\}. 
\]
The natural assumptions are $1\leq t\leq r$ and $n\geq t+1$.
Geometrically, when $\ell\in G^r_{d}(C)$ is very ample and thus induces an embedding $C\hookrightarrow \mathbb{P}^r$, an element $D\in C_n$ such that $h^0\left( C, \ell\otimes \mathcal{O}_C (-D) \right) \geq r+1-t$ corresponds to an $n$-secant $(t-1)$-plane to $C\subset \mathbb{P}^r$. If $t=1$, then $\mathfrak{S}^{r,1}_{g,d} (C)$ is the variety of symmetric $n$-tuples of points in $C$ mapping to an $n$-fold point in some projective model of $C$. On the other end of the range, if $t=r$, then $\mathfrak{S}^{r,r}_{g,d} (C)$ parametrizes symmetric $n$-tuples of points which become linearly dependent in some embedding $C\subset \mathbb{P}^r$.

When  $\rho(g,r,d)= (n-t)(r+1-t)-1$, the variety $\mathfrak{S}^{r,t}_{g,d} (C)$ is expected to be a  divisor in $C_n$. 
More generally, the incidence correspondence 
\[
\Sigma^{r,t}_{g,d}(C) := \{ (D,\ell) \in C_n \times G^r_d(C) \, | \, h^0(C, \ell \otimes \mathcal{O}_C(-D))\geq r+1-t \}
\]
is a determinantal subvariety of $C_n \times G^r_d(C)$ of expected dimension $n+\rho(g,r,d)-(n-t)(r+1-t)$.  
From \cite[\S 2]{MR2439633}, $\Sigma^{r,t}_{g,d}(C)$ is either empty, or pure of the expected dimension.
If $\rho(g,r,d)= (n-t)(r+1-t)-1$, the locus $\Sigma^{r,t}_{g,d}(C)$ is known to be non-empty in each of the following cases
\begin{eqnarray}
\label{nonempty}
\begin{array}{rl}
\mbox{(i)} & \mbox{$g-d+r=1$ and $n\geq (n-t)(r+1-t)$ (\cite[pg.~356]{MR770932});}\\
\mbox{(ii)} & \mbox{$g-d+r\geq 1$, $n\geq (n-t)(r+1-t)+r-t$, and $d\geq 2n-1$ (\cite{MR1104787});}\\
\mbox{(iii)} & \mbox{$g-d+r\geq 1$ and $n\leq (n-t)(r+1-t)$ (\cite[Theorem 0.5]{MR2439633});}\\
\mbox{(iv)} & \mbox{$g-d+r\geq 1$ and $t=r$.}
\end{array}
\end{eqnarray}
Case (iv) follows from the others. Indeed if $t=r$, then $\Sigma^{r,t}_{g,d}(C)$ is non-empty if $h^0(C,\mathcal{O}_C(D+E))\geq r+1$ for some $D\in C_n$ and $E\in C_{d-n}$.
Equivalently, $h^0(C,K_C\otimes \mathcal{O}_C(-D-E))\geq g-d+r$.
Since $K_C\otimes \mathcal{O}_C(-E) \in G^{g-d+n-1}_{2g-2-d+n}(C)$ for a general $E\in C_{d-n}$, we have that 
$\Sigma^{r,r}_{g,d} (C)= \Sigma^{g-d+n-1,n-r}_{g, 2g-2-d+n} (C)$. Finally, $\Sigma^{g-d+n-1,n-r}_{g, 2g-2-d+n} (C)$ always satisfies either (i) or (iii) in (\ref{nonempty}).

We conclude that $\Sigma^{r,t}_{g,d}(C)$ is pure of dimension $n-1$ when $\rho(g,r,d)= (n-t)(r+1-t)-1$ and one of the conditions in (\ref{nonempty}) holds. Moreover, the first projection $\pi_1\colon \Sigma^{r,t}_{g,d}(C) \rightarrow \mathfrak{S}^{r,t}_{g,d} (C) \subset C_n$ is generically finite on its image.
In other words, given a generic element $D=p_1+\cdots+p_n\in C_n$, there is at most a finite number of linear series $\ell$ in $G^r_d(C)$ such that $h^0(C, \ell \otimes \mathcal{O}_C(-D))\geq r+1-t$.
To verify this, by letting the $n$ points $p_i$ collide together, one shows that for any point $x$ in $C$ there is at most a finite number of linear series $\ell$ in $G^r_d(C)$ such that $h^0(C,\ell \otimes \mathcal{O}_C(-nx))\geq r+1-t$. 
This follows from \cite{MR985853} (see the discussion in \S \ref{nn}). Hence, $\mathfrak{S}^{r,t}_{g,d} (C)$ is a divisor in $C_n$ when $\rho(g,r,d)= (n-t)(r+1-t)-1$ and one of the conditions in (\ref{nonempty}) is satisfied.

Let $C$ be a {\it very general} curve, that is, a curve corresponding to a point outside the union of countably many subvarieties of the moduli space of curves of genus $g$.
The N\'eron-Severi group $N_\mathbb{Q}^1(C_n)$ is generated by the classes $\theta$ and $x$, where $\theta$ is the pull-back of the class of the theta divisor via the natural map $C_n\rightarrow {\rm Pic}^n(C)$, and $x$ is the class of the locus $p+C_{n-1}\subset C_n$ of divisors of $C$ containing a fixed general point $p\in C$ (\cite[pg.~359]{MR770932}). 
Our first result is an explicit expression for the class of the divisor $\mathfrak{S}^{r,t}_{g,d} (C)$ in $N_\mathbb{Q}^1(C_n)$.

\begin{thm}
\label{Secresult}
Fix $g,d\geq 2$, $1\leq t\leq r$, and $n\geq t+1$ such that $\rho(g,r,d)= (n-t)(r+1-t)-1$. Let $C$ be a very general curve of genus $g$, and let $s:=g-d+r$. Assume that one of the conditions in (\ref{nonempty}) is satisfied.
Then
$\mathfrak{S}^{r,t}_{g,d} (C)$ 
is an effective divisor in $C_n$, and its class in $N_\mathbb{Q}^1(C_n)$ is
\[
\left[\mathfrak{S}^{r,t}_{g,d} (C)\right] = g! \cdot n \frac{\prod\limits_{i=2}^t \dfrac{i! \cdot (n-i)}{(s-1+i)!} 
\cdot \prod\limits_{j=2}^{r+1-t} \dfrac{j! (n+j)!}{(s+n-1+j)! (n-t-1+j)!(n-1+j)}} {g(s-1)!(s+n-1)!(t-1)!(r-t)!} 
 \left(\theta -\frac{g}{n} x   \right).
\]
\end{thm}

It follows that the class $\theta -(g/n) x $ is $\mathbb{Q}$-effective for all values of $g$ and $n$ that satisfy the hypotheses of the  theorem. In particular, these  include all cases $g\geq 4$ and $g/2 \leq n \leq g-2$. 
Moreover, we remark that the classes of all the divisors $\mathfrak{S}^{r,t}_{g,d} (C)$  obtained by varying $r,t$, and $d$ (while keeping $g$ and $n$ fixed) lie on the {\it same ray} spanned by $\theta -(g/n) x $ inside $N_\mathbb{Q}^1(C_n)$.

As an example, consider the case $r=1$. For $n=2d-g\geq 2$, Theorem \ref{Secresult} recovers the formula 
\begin{eqnarray}
\label{r=1}
\left[\mathfrak{S}^{1,1}_{g,d} (C)\right] = \frac{n}{g} \binom{g}{g-d} \left(\theta -\frac{g}{n} x   \right)
\end{eqnarray}
from \cite[Corollary 3.2]{MR3102167}. Given an integer $1\leq m \leq g/2-1$, (\ref{r=1}) gives a divisor class on $C_n$ for $n=g-2m$.
Remarkably, in the case $n=g-2$, the divisor $\mathfrak{S}^{1,1}_{g,g-1} (C)$ is {\it extremal} in the cone of effective divisor classes of $C_{g-2}$ (\cite{Mus}). Moreover, in the case $r=t=2$ and $n=g-3\geq 3$, Theorem \ref{Secresult} recovers the formula for the class of the divisor $\mathfrak{S}^{2,2}_{g,g} (C)$ in $C_{g-3}$ first studied in \cite{MR3031571}.

It is natural to ask when the class of $\mathfrak{S}^{r,t}_{g,d} (C)$ is extremal in the {\it cone of effective divisor classes}  ${\rm Eff}(C_n)$  inside the two-dimensional space $N_\mathbb{Q}^1(C_n)$.
The diagonal class $\delta_C := -\theta + (g+n-1)x$ spans an extremal ray of ${\rm Eff}(C_n)$ (\cite[\S 6]{MR1149124}), while the other extremal ray is generally not known. 
The class of $\mathfrak{S}^{r,t}_{g,d} (C)$ gives us a bound for the slope of ${\rm Eff}(C_n)$ in the fourth quadrant of the $(\theta,x)$-plane for all values of $g$ and $n$ satisfying the hypotheses of Theorem \ref{Secresult}. For all such values in the range $\sqrt{g} \leq n \leq g-3$, this bound improves the one given by the $\mathbb{Q}$-effective class $\theta-\lfloor g/n \rfloor x$ (\cite[\S 7]{MR1149124}). In the cases $n=g-2$ (\cite{Mus}) and $n=g$ (\cite{MR1149124}), the ray spanned by $\theta-(g/n) x$ is known to be extremal. Better bounds are known for $n<\sqrt{g}$, while the class $\theta$ is extremal for $n\geq g+1$ (\cite[\S5, \S 7]{MR1149124}). The case $n=g-1$ is excluded by the hypotheses of Theorem \ref{Secresult}, but a full description of ${\rm Eff}(C_{g-1})$ is given in \cite{MR2797351}.

In \S \ref{Sec} we prove a result stronger than Theorem \ref{Secresult}. By letting the moduli of the curve $C$ in Theorem \ref{Secresult} vary, we consider the divisor $\mathfrak{S}^{r,t}_{g,d}$ inside the moduli space $\M_{g,n}$ defined as the locus of pointed curves $[C,x_1,\dots, x_n]$ such that the divisor $x_1+\cdots +x_n$ imposes at most $t$ independent conditions on some linear series in $G^r_d(C)$ (see Theorem \ref{Secclass}). In \S \ref{Sec}, we study the class of its closure inside the moduli space $\Mbar_{g,n}$ of stable pointed curves. As an auxiliary computation, in \S \ref{PBNdiv} we calculate the class of the pointed Brill-Noether divisor in $\Mbar_{g,1}$ obtained by letting all marked points in a general element of $\mathfrak{S}^{r,t}_{g,d}$ collide together. In order to do this, we start by counting special Brill-Noether points on a general curve in \S \ref{nn}. 
Finally, in \S \ref{symprod} we deduce Theorem \ref{Secresult} from the  class of the divisor $\overline{\mathfrak{S}}^{r,t}_{g,d}$ in $\Mbar_{g,n}$.

\vskip4pt

\noindent {\bf Acknowledgments.}
I would like to thank Gavril Farkas for many helpful discussions related to varieties of secant divisors.


\section{Brill-Noether special points on the general curve}
\label{nn}

In this section, we count linear series on a general curve with exceptional vanishing at an arbitrary point. Given a linear series $\ell=(L,V)\in G^r_d(C)$ on a general curve $C$, the {\it vanishing sequence} of $\ell$ at some point $p\in C$ is defined as the sequence $a^\ell(p) \colon 0\leq a_0 < \cdots < a_r\leq d$ of distinct vanishing orders of sections in $V$ at $p$.
 
Fix $g\geq 2$, two positive integers $r$, $d$, and an index $a: 0\leq a_0 < \cdots < a_r \leq d$ such that $\rho(g,r,d,{a}):= g-(r+1)(g-d+r)-\sum_i(a_i -i) = -1$. 
From \cite{MR985853}, a general curve $C$ of genus $g$ admits at most a finite number of linear series $\ell\in G^r_d(C)$ having vanishing sequence $a^\ell(x)=a$ at some point $x\in C$. From \cite{PointedCastelnuovo}, the number of pairs $(x,\ell) \in C\times G^r_d(C)$ such that $a^{\ell}(x)={a}$ is equal to 
\begin{eqnarray}
\label{PCa}
n_{g,r,d,a}:=
 g! \sum_{0\leq j_1<j_2\leq r } \left( (a_{j_2}-a_{j_1})^2-1  \right) \frac{\prod_{0\leq i<k\leq r}(a_k-\delta_k^{j_1}-\delta_k^{j_2}-a_i+\delta_i^{j_1}+\delta_i^{j_2})}{\prod_{i=0}^r (g-d+r+a_i-\delta_i^{j_1}-\delta_i^{j_2})!}.
\end{eqnarray}
Here, $\delta^i_j$ is the Kronecker delta, and we set $1/n!=0$, if $n<0$.  In the following, we will use the above formula in the case 
$a=(0,\dots, t-1, n, \dots, n+r-t)$. We will also need to count linear series satisfying an exceptional Brill-Noether condition on a general pointed curve.

\begin{prop}
\label{Prop}
Fix $g,d\geq 2$, $1\leq t\leq r$, and $n\geq t+1$ such that $\rho(g,r,d)= (n-t)(r+1-t)-1$. 

i) Let $C$ be a general curve of genus $g$. The number of pairs $(x,\ell) \in C\times G^r_d(C)$ such that $h^0(C,\ell\otimes \mathcal{O}_C(-n\cdot x))\geq r+1-t$ is
\[
n_{g,r,d,a} = g! \cdot n (n^2-1) \frac{\prod\limits_{i=2}^t \dfrac{i! \cdot (n-i)}{(s-1+i)!} 
\cdot \prod\limits_{j=2}^{r+1-t} \dfrac{j! (n+j)!}{(s+n-1+j)! (n-t-1+j)!(n-1+j)}} {(s-1)!(s+n-1)!(t-1)!(r-t)!} 
\]
with $a=(0,\dots, t-1, n, \dots, n+r-t)$ and $s:=g-d+r$.

ii) Let $(C,p)$ be a general pointed curve of genus $g$. For $1\leq \delta \leq n$, there exists a finite number of pairs $(x,\ell) \in C\times G^r_{d}(C)$ such that
$h^0(C, \ell\otimes \mathcal{O}_C(-\delta x - (n-\delta)p))\geq r+1-t$.
Their number is
\begin{eqnarray*}
T_{g,d}^{r,t}(\delta) :=n_{g,r,d,a}\cdot \frac{\delta (n\delta -1)}{n (n^2-1)}
\end{eqnarray*}
with $n_{g,r,d,a}$ as in $(i)$.
\end{prop}

\begin{proof}
Part $(i)$ follows from (\ref{PCa}). The formula for $n_{g,r,d,a}$ with $a=(0,\dots, t-1, n, \dots, n+r-t)$ has only one non-zero summand, that is, the contribution given by $j_1=0$ and $j_2=t$.

The set-up of the proof of part $(ii)$ is similar to that of the proof of (\ref{PCa}) from \cite{PointedCastelnuovo}. 
Choose $m$ such that $h^1(C,L\otimes \mathcal{O}_C(mp))=0$, for every $L\in {\rm Pic}^d(C)$. The natural maps
\[
H^0(C,L\otimes \mathcal{O}_C(mp)) \rightarrow H^0(C,L\otimes \mathcal{O}_C(mp)|_{(m+n-\delta)p+\delta x}) \twoheadrightarrow H^0(C,L\otimes \mathcal{O}_C(mp)|_{mp})
\]
globalize to
\[
\pi^*(\mathcal{E}) \rightarrow \mu_*(\nu^* \mathcal{L}\otimes \mathcal{O}_{(m+n-\delta)\Gamma_p+\delta\Delta})=:\mathcal{M}_1 \twoheadrightarrow
 \mu_*(\nu^* \mathcal{L}\otimes \mathcal{O}_{m\Gamma_p})=:\mathcal{M}_0
\]
as maps of vector bundles over $C\times {\rm Pic}^{d+m}(C)$. Here $\mathcal{L}$ is the Poincar\'e bundle on $C\times {\rm Pic}^{d+m}(C)$ normalized to be trivial on $\{p\}\times {\rm Pic}^{d+m}(C)$, the map $\pi\colon C\times {\rm Pic}^{d+m}(C)\rightarrow {\rm Pic}^{d+m}(C)$ is the second projection, and $\mathcal{E}$ is the vector bundle of rank $d+m-g+1$ defined as $\mathcal{E}:=\pi_*(\mathcal{L})$. The maps $\mu\colon C\times C\times {\rm Pic}^{d+m}(C)\rightarrow C\times {\rm Pic}^{d+m}(C)$ and $\nu\colon C\times C\times {\rm Pic}^{d+m}(C)\rightarrow C\times {\rm Pic}^{d+m}(C)$ are the projections respectively onto the first and third, and the second and third factors. Finally, $\Delta$ and $\Gamma_p$ are the pull-backs to $C\times C\times {\rm Pic}^{d+m}(C)$ respectively of the diagonal in $C\times C$ and of the divisor $C\times \{ p\}\subset C\times C$.  
Note that $\mathcal{M}_0$ is a trivial bundle of rank $m$ over $C\times {\rm Pic}^{d+m}(C)$. 

Linear series $\ell\in G^r_d(C)$ satisfying the condition $h^0(C,\ell\otimes\mathcal{O}_C(-(n-\delta)p-\delta x))\geq r+1-t$ are complete (as are linear series in part $(i)$).
Thus, we are interested in the pairs $(x, L)$ in $C\times {\rm Pic}^d(C)$ such that $h^0(C,L\otimes\mathcal{O}_C(-(n-\delta)p-\delta x))\geq r+1-t$ 
and $h^0(C,L)\geq r+1$. This is the locus in $C\times {\rm Pic}^{d+m}(C)$ where the ranks of the maps
\begin{eqnarray}
\label{varphi}
\varphi_1 \colon \pi^*(\mathcal{E}) \rightarrow \mathcal{M}_1, \quad \quad \quad \quad \quad \varphi_0 \colon \pi^*(\mathcal{E}) \rightarrow \mathcal{M}_0
\end{eqnarray}
are respectively bounded by ${\rm rank} (\varphi_1) \leq d+m-g-r+t$, and ${\rm rank} (\varphi_0) \leq d+m-g-r$. 
The finiteness of the locus of pairs considered in part $(i)$ implies the finiteness of this locus. 
Hence, we can apply the Fulton-Pragacz determinantal formula for flag bundles \cite[Theorem 10.1]{MR1154177}.

Let us compute the Chern classes of the bundles in (\ref{varphi}). Let $\eta_i$ be the pull-back of the class of a point in $C$ via the projection on the $i$ factor $\pi_i \colon C\times C \times {\rm Pic}^{d+m}(C)\rightarrow C$, for $i=1,2$, and let $\theta$ be the pull-back of the class of the theta divisor in $ {\rm Pic}^{d+m}(C)$ via $\pi_3 \colon C\times C \times {\rm Pic}^{d+m}(C)\rightarrow {\rm Pic}^{d+m}(C)$. 
Finally, given a symplectic basis $\delta_1, \dots, \delta_{2g}$ for $H^1(C, \mathbb{Z})\cong H^1( {\rm Pic}^{d+m}(C), \mathbb{Z})$, let $\delta^i_s :=\pi_i^*(\delta_s)$ for $i=1,2,3$, and define
\[
\gamma_{i,j} := - \sum_{s=1}^g (\delta^j_s \delta^i_{g+s} - \delta^j_{g+s} \delta^i_s).
\]
The following relations will ease the computation. Note that $\eta_i^2= \eta_i\cdot \gamma_{i,j} = \gamma_{i,j}^3= \theta^{g+1} =0$, for any $i=1,2$, and $j=2,3$. Moreover, it is easy to see that $\gamma_{1,2}^2 = -2g \eta_1 \eta_2$, $\gamma_{i,3}^2=-2\eta_i \theta$ for $i=1,2$, and $\gamma_{i,j}\gamma_{j,3}=\eta_j \gamma_{i,3}$ for $\{i,j\}=\{1,2\}$.
From \cite[\S VIII.2]{MR770932}, we have $c_t(\mathcal{E})=e^{-t\theta}$, and
\begin{align*}
ch(\nu^*\mathcal{L}) & =  1+ (d+m)\eta_2 + \gamma_{2,3} - \eta_2\theta, &
ch(\mathcal{O}_{(m+n-\delta)\Gamma_p+\delta\Delta}) &= 1-e^{-(\delta \eta_1 + \delta \gamma_{1,2} + (m+n) \eta_2)}.
\end{align*}
Via the Grothendieck-Riemann-Roch formula, we have
\begin{eqnarray*}
ch(\mathcal{M}_1) &=& \mu_*((1+(1-g)\eta_2)\cdot ch(\nu^*\mathcal{L} \otimes \mathcal{O}_{(m+n-\delta) \Gamma_p+\delta\Delta}))\\
&=& (m+n)+ \eta_1(g\delta^2+\delta(d-g+1-n)) +\delta \gamma_{1,3} -\delta \eta_1 \theta,
\end{eqnarray*}
whence we compute the Chern polynomial
\[
c_t(\mathcal{M}_1)= 1+\eta_1(g\delta^2+\delta(d-g+1-n))+\delta \gamma_{1,3} + (\delta-\delta^2)\eta_1 \theta.
\]
Let $c_t ^{(0)}:=c_t(\mathcal{M}_0-\mathcal{E})=e^{t\theta}$, and $c_t ^{(1)}:=c_t(\mathcal{M}_1-\mathcal{E})$, that is,
\[
c^{(1)}_1 = \eta_1 (g\delta^2+\delta(d-g+1-n))+\delta \gamma_{1,3} + \theta
\]
and 
\[
c^{(1)}_j = \frac{\theta^j}{j!}+\eta_1\theta^{j-1}\left( \frac{g\delta^2+\delta(d-g+1-n)}{(j-1)!}+\frac{\delta-\delta^2}{(j-2)!} \right) +\frac{\delta}{(j-1)!}\gamma_{1,3}\theta^{j-1}
\]
for $j\geq 2$. 
Using the Fulton-Pragacz formula, the degree of the class of the locus of pairs $(x,L)$ in $C\times {\rm Pic}^d(C)$ such that $h^0(C, L\otimes\mathcal{O}_C(-\delta x - (n-\delta)p))\geq r+1-t$ and $h^0(C,L)\geq r+1$ is 
\[
T_{g,d}^{r,t}(\delta) := \deg  \left[
\begin{array}{cccccc}
c^{(1)}_{g-d+r+n-t} &  & \cdots & & & c^{(1)}_{g-d+2r+n-t}\\
\vdots &  \ddots & && &\vdots \\
c^{(1)}_{g-d+n} & \cdots & c^{(1)}_{g-d+r+n-t} & \cdots & & c^{(1)}_{g-d+r+n}\\
c^{(0)}_{g-d+t-1} & & \cdots & c^{(0)}_{g-d+r} &\cdots & c^{(0)}_{g-d+r+t-1}\\
\vdots &  &  &  & \ddots &\vdots \\
c^{(0)}_{g-d} & &  & \cdots & & c^{(0)}_{g-d+r}
\end{array}
\right].
\]
The determinant of the above $(r+1)\times (r+1)$ matrix can be computed as follows. Since $\eta_1^2=\eta_1\gamma_{1,3}=\gamma_{1,3}^3=0$, we deduce that $T_{g,d}^{r,t}(\delta)$ is a polynomial in $\delta$ of degree at most two. Hence, it is enough to determine $T_{g,d}^{r,t}(\delta)$ for three different values of $\delta$. 

If $\delta=0$, then $T_{g,d}^{r,t}(0)=0$: indeed, the above determinant is a multiple of $\theta^{g+1}$, hence zero. Geometrically, the case $\delta=0$ corresponds to counting linear series $\ell\in G^r_d(C)$ 
such that $a^\ell(p)=a=(0,\dots, t-1, n, \dots, n+r-t)$ at a general point $p$. Since $\rho(g,r,d,a)=-1$,
such a locus is empty by \cite{MR985853}. Moreover, if $\delta=n$, then $T_{g,d}^{r,t}(n)=n_{g,r,d,a}$ by $(i)$.

Let us compute $T_{g,d}^{r,t}(1)$. Using again $\eta_1^2=\eta_1\gamma_{1,3}=\gamma_{1,3}^3=0$, the determinant $T_{g,d}^{r,t}(1)$ is by linearity  equal to the sum of the determinants of the matrices containing {\it exactly one} row with only coefficients of type
$\eta_1\theta^{j-1}(d+1-n)/(j-1)!$
and all other rows with coefficients of type $\theta^j/j!$, 
plus the determinants of the matrices containing {\it exactly two} rows with only coefficients of type
$\gamma_{1,3}\theta^{j-1}/(j-1)!$
and all other rows with coefficients of type $\theta^j/j!$. The determinants of such matrices can be expressed in terms of the following variation of the Vandermonde determinant
\begin{eqnarray*}
\Delta(b_0, \dots , b_r):=
\left[
\begin{array}{ccc}
\frac{1}{(b_r-r)!} & \cdots  & \frac{1}{b_r!}\\
\vdots &  \ddots & \vdots \\
\frac{1}{(b_0-r)!} &   \cdots & \frac{1}{b_0!}
\end{array}
\right]
= \frac{\prod_{i<j} (b_j - b_i)}{ \prod_{k=0}^r b_k!}.
\end{eqnarray*}
Fixing $s=g-d+r$, there are only two non-zero determinants in the expansion of $T_{g,d}^{r,t}(1)$, that is,
\begin{eqnarray*}
T_{g,d}^{r,t}(1) &=& (d+1-n) \cdot \deg ( \eta_1\theta^{g} )\cdot \Delta(s, \dots, s+t-1, s+n-1, s+n+1, \dots, s+n+r-t)  \\
&& {}+ \deg(\gamma^2_{1,3}\theta^{g-1}) \cdot \Delta(s, \dots, s+t-1, s+n-1, s+n, s+n+2, \dots, s+n+r-t) .
\end{eqnarray*}
Expanding, we have
\begin{multline*}
T_{g,d}^{r,t}(1)= g! \prod_{i=0}^{t-1}\frac{ i! (n-1-i) }{(s+i)! }\prod_{j=0}^{r-t} \frac{1}{(s+n+j)!}\\
{}\times \left( \frac{(d+1-n)(s+n)}{(r-t)!} \prod_{j=1}^{r-t} \frac{(n+j)!(j+1)!}{(n+j-t)!}  
 -2  \frac{(s+n+1)(s+n) }{(r-t-1)!} \frac{n!}{(n-t)!}\prod_{j=2}^{r-t} \frac{(n+j)!(j+1)!}{(n+j-t)!}  \right).
\end{multline*}
Note that we can write $d=(s+1)r +\rho(g,r,d)= (s+1)r +(n-t)(r+1-t)-1$.
Thus, the above expression can be simplified as follows
\[
T_{g,d}^{r,t}(1) =  \frac{n_{g,r,d,a}}{n(n+1)}.
\]
Interpolating the values of $T_{g,d}^{r,t}(\delta)$ for $\delta=0,1,n$, we deduce the statement in $(ii)$.
\end{proof}

\section{Pointed Brill-Noether divisors}
\label{PBNdiv}

In this section, we consider family of curves with a marked Brill-Noether special point. Let $g>2$, $r,d$, and $a: 0\leq a_0 < \cdots < a_r \leq d$ be such that $\rho(g,r,d,{a}):= g-(r+1)(g-d+r)-\sum_i(a_i -i) = -1$. Consider the locus $\mathcal{M}_{g,d}^r(a)$ in $\mathcal{M}_{g,1}$ of pointed curves $[C,x]$ with a linear series $\ell\in G^r_d(C)$ having vanishing sequence $a^\ell(x)\geq a$. From \cite{MR985853}, the class of the closure of a divisor of type $\mathcal{M}_{g,d}^r(a)$ in the moduli space of stable pointed curves $\overline{\mathcal{M}}_{g,1}$ lies in the two-dimensional cone generated by the pull-back of the class of the Brill-Noether divisor in $\overline{\mathcal{M}}_g$
\[
\mathcal{BN} := (g+3)\lambda - \frac{g+1}{6}\delta_{\rm irr} - \sum^{g-1}_{i=1} i(g-i) \delta_i
\]
and the class of the closure of the Weierstrass divisor $\M^1_{g,g}(0,g)$
\[
\mathcal{W} := -\lambda + \binom{g+1}{2} \psi - \sum_{i=1}^{g-1} \binom{g-i+1}{2} \delta_i.
\]
Here, $\lambda$ is the pull-back from $\Mbar_g$ of the first Chern class of the Hodge bundle, $\psi$ is the cotangent class at the marked point, $\delta_{\rm irr}$ is the class of the closure of the locus of nodal irreducible curves, and $\delta_i$ is the class of the closure of the locus of curves with a pointed component of genus $i$ attached at a component of genus $g-i$.

It follows that we can write $[\overline{\mathcal{M}}_{g,d}^r(a)] = \mu\cdot \mathcal{BN}+ \nu \cdot \mathcal{W} \in {\rm Pic}(\overline{\mathcal{M}}_{g,1})$ for some positive coefficients $\mu$ and $\nu$. Such coefficients can be computed via test curves: from \cite[Corollary 1]{PointedCastelnuovo}, we have
\begin{eqnarray}
\label{munu}
\mu = - \frac{n_{g,r,d,a} }{2(g^2-1)} + \frac{1 }{4\binom{g-1}{2}}\sum_{i=0}^r n_{g-1,r,d,(a_0+1-\delta^i_0,\dots,a_r+1-\delta^i_r)}
\quad\quad\mbox{and}\quad\quad
\nu = \frac{n_{g,r,d,a} }{g(g^2-1)}.
\end{eqnarray}
Here, $\delta^i_j$ is the Kronecker delta. We apply these formulae in the case $a=(0,\dots, t-1, n, \dots, n+r-t)$, using the expression for $n_{g,r,d,a}$ from Proposition \ref{Prop} $(i)$.

\begin{lem}
\label{PBN}
Fix $g,d\geq 2$, $1\leq t\leq r$, and $n\geq t+1$ such that $\rho(g,r,d)= (n-t)(r+1-t)-1$. Let $a=(0,\dots, t-1, n, \dots, n+r-t)$, $s:=g-d+r$, and $e:=d-g-t+1$.
The class of the closure of the divisor
\[
\mathcal{M}_{g,d}^r(a) := \{[C,x] \in \mathcal{M}_{g,1}  : \exists \, \ell\in G^r_d(C) \,\,\mbox{with}\,\, h^0(C,\ell\otimes \mathcal{O}_C(-n\cdot x))\geq r+1-t \}
\]
is 
\[
\left[\overline{\mathcal{M}}_{g,d}^r(a) \right] = \frac{n_{g,r,d,a}}{g(g^2-1)} \left( \sigma\cdot \mathcal{BN}+ \mathcal{W} \right)
\]
with $n_{g,r,d,a}$ as in Proposition \ref{Prop}, and
\[
\sigma:=\frac{ \left(t e +g +1\right) (n-t) \left[  (n-t) (3 t e-g-1)  +2 (g+1)  (e-t)\right] -(d+1) (g+1)^2  (d-2 g+1)}{2(g-2)s(n+s)t(n-t)(r+1-t)(n+r+1-t)}
\]
for $g\geq 3$, while $\sigma=0$ for $g=2$.
\end{lem}

\begin{proof}
The statement follows from (\ref{munu}).
Note that the formula for $\mu$ reduces to
\begin{eqnarray*}
\mu = - \frac{n_{g,r,d,a} }{2(g^2-1)} + \frac{1}{4\binom{g-1}{2}} \left( n_{g-1,r,d,b} + n_{g-1,r,d,c} \right)
\end{eqnarray*}
where $b = (0,2,\dots, t, n+1, \dots, n+r-t+1)$  and $c = (1,\dots, t, n, n+2 \dots, n+r-t+1).$
The formula (\ref{PCa}) for $n_{g-1,r,d,b}$ has non-zero contributions only for $(j_1,j_2)\in \{(0,1),(0,t),(1,t)\}$. Similarly, the only non-zero contributions to $n_{g-1,r,d,c}$ are for $(j_1,j_2)\in \{(0,t),(0,t+1),(t,t+1)\}$. We have
\begin{eqnarray*}
n_{g-1,r,d,b} &=& \frac{n_{g,r,d,a}}{n^2-1} \cdot \left(\frac{(t+1)(t-1)(s-1)(s+1)(n+r-t+2)(n+r-t)}{g(n-t)(s+n)n(r-t+1)} \right.\\
		    & &{}+ \frac{((n+1)^2-1)(t+1)(n+1)(s-1)(n+r-t+2)}{2gn(n+2)}\\
		    & &{}+\left. \frac{((n-1)^2-1)(t-1)(n-1)(s+1)(n+r-t)}{2g(n-2)n} \right),\\
n_{g-1,r,d,c} &=& \frac{n_{g,r,d,a}}{n^2-1} \cdot \left( \frac{((n-1)^2-1)(n-1)(n-t-1)(r-t+2)(s+n-1)}{2gn(n-2)} \right.\\
                    & & {}+ \frac{((n+1)^2-1)(n+1)(n-t+1)(r-t)(s+n+1)}{2g(n+2)n} \\
                    & & {}+\left. \frac{(n-t-1)(n-t+1)(r-t+2)(r-t)(s+n-1)(s+n+1)}{g\cdot s\cdot t\cdot n(n+r-t+1)} \right).
\end{eqnarray*}
Modulo the identities $g=(r+1)s + \rho(g,r,d)$ and $d=(s+1)r+ \rho(g,r,d)$, the resulting formula for $\mu$ is equivalent to the total coefficient of  $\mathcal{BN}$ in the statement.
\end{proof}

\section{The divisor  \texorpdfstring{$\mathfrak{S}^{r,t}_{g,d}$}{S}}
\label{Sec}

Fix $g,d\geq 2$, $1\leq t\leq r$, and $n\geq t+1$ such that $\rho(g,r,d)= (n-t)(r+1-t)-1$. Let $C$ be a general curve of genus $g$. After the discussion in the introduction,
the locus $\mathfrak{S}^{r,t}_{g,d}(C)$ is a divisor in $C_n$ when one of the conditions in (\ref{nonempty}) is satisfied. By varying the moduli of the curve $C$, in this section we study the divisor in $\M_{g,n}$ of pointed curves $[C, x_1, \dots, x_n]$ 
such that $x_1+\cdots + x_n$ imposes at most $t$ independent conditions on some linear series of type $\mathfrak{g}^r_d$, that is,
\[
\mathfrak{S}^{r,t}_{g,d} := \left\{ [C,x_1,\dots,x_n] : \exists \, \ell \in G^r_{d}(C) \,\,\mbox{with}\,\, h^0\left( C,\ell \otimes \mathcal{O}_C\left({} - x_1-\cdots -x_n \right) \right) \geq r+1-t   \right\}.
\]

We compute the main coefficients of the class of the closure of $\mathfrak{S}^{r,t}_{g,d}$ in $\Mbar_{g,n}$.
Let us fix the notation for divisor classes on $\Mbar_{g,n}$. For $0\leq i \leq g-1 $ and $J\subset \{1,\dots, n \}$, let $\delta_{i:J}$ be the class of the closure of the locus of curves with a component of genus $i$ attached at a component of genus $g-i$, and the marked points on the component of genus $i$ are exactly those corresponding to $J$ (assume $|J|\geq 2$ if $i=0$). Denote by $\delta_{i:j}$ the sum of the classes $\delta_{i:J}$ such that $|J|=j$. Finally, $\lambda$ and $\delta_{\rm irr}$ are the pull-back of the analogous classes on $\Mbar_g$, and $\psi_i$ is the cotangent class at the point $i$, for $i=1,\dots,n$. Modulo the identity $\delta_{i,J}=\delta_{g-i, J^c}$, the classes $\lambda$, $\psi_i$, $\delta_{\rm irr}$, and $\delta_{i,J}$ form 
a basis of ${\rm Pic}\left(\Mbar_{g,n} \right)$ for $g\geq 3$, and generate ${\rm Pic}\left(\Mbar_{g,n} \right)$ for $g=2$.

\begin{thm}
\label{Secclass}
Fix $g,d\geq 2$, $1\leq t\leq r$, and $n\geq t+1$ such that $\rho(g,r,d)= (n-t)(r+1-t)-1$.
Assume that one of the conditions in (\ref{nonempty}) holds.
The class of the closure of the divisor $\mathfrak{S}^{r,t}_{g,d}$  is
\[
\left[ \overline{\mathfrak{S}}^{r,t}_{g,d}\right]  =  \frac{n_{g,r,d,a}}{g(g^2-1)} \left(c_\lambda \lambda + c_\psi \sum_{i=1}^{n} \psi_i -c_{ \rm irr} \delta_{\rm irr} - \sum_{j=2}^{ n} c_{0:j} \delta_{0:j} - \sum_{i=1}^{g-1}\sum_{j=0}^{\lfloor n \rfloor} c_{i:j} \delta_{i:j} \right) \in {\rm Pic}(\Mbar_{g,n})
\]
where $n_{g,r,d,a}$ is as in Proposition \ref{Prop}, and, for $\sigma$ as in Lemma \ref{PBN}, one has
\begin{align*}
c_\lambda &= \sigma(g+3)-1, &
c_{\rm irr} &= \sigma \cdot \frac{g+1}{6}, &
c_\psi &= \frac{(g+1)(g+n)}{2n(n+1)}, 
\end{align*}
\begin{align*} 
c_{0:j} & =  \frac{j(g+1)(n^2 +jgn -jn -g)}{2n(n^2-1)}, 
&
c_{i:0} &=  \sigma\cdot i(g-i)+ \frac{i(i+1)}{2}. 
\end{align*}
\end{thm}

\begin{ex}
We recover the case $r=1$ and $d=g$ studied in \cite{MR1953519},
and in general the case $r=1$ studied in \cite[\S 4.3]{MR2530855}.
Moreover, we recover the case $\rho(g,r,d)=0$ studied in \cite[\S 4.2]{MR2530855}, and the case $r=t=2$ and $d=g$ studied in \cite{MR3031571}. 
\end{ex}

\begin{rem}
The divisor $\overline{\mathfrak{S}}^{1,1}_{g,g-1}$ is rigid and extremal in $\Mbar_{g,g-2}$ (\cite{MR3102167}). Moreover, 
the divisor $\overline{\mathfrak{S}}^{1,1}_{g,g}$ is rigid and extremal in $\Mbar_{g,g}$,  for $2\leq g \leq 11$ (\cite{MR3093504}).
 It is natural to ask whether there are other cases in which the divisor $\overline{\mathfrak{S}}^{r,t}_{g,d}$ is rigid or extremal.
\end{rem}

\begin{proof}[Proof of Theorem \ref{Secclass}]
The proof follows the strategy from \cite[\S 3]{MR1953519} and \cite[\S 4]{MR2530855}.
Let us consider the divisor $(\mathfrak{S}^{r,t}_{g,d})^{ n}$ in $\M_{g,1}$ obtained by letting the $n$ marked points in a general element in $\mathfrak{S}^{r,t}_{g,d}$ collide together, that is,
\begin{eqnarray*}
(\mathfrak{S}^{r,t}_{g,d})^{ n} := \left\{ [C,x]\in \M_{g,1} : \exists \, \ell \in G^r_{d}(C) \,\,\mbox{with}\,\, h^0\left( C,\ell\otimes\mathcal{O}_C\left( -nx \right) \right) \geq r+1-t \right\}.
\end{eqnarray*}
This divisor coincides with the pointed Brill-Noether divisor $\M^r_{g,d}(a)$ studied in \S \ref{PBNdiv}, where $a=(0,\dots, t-1, n, \dots, n+r-t)$. 
From \cite[Theorem 2.8]{MR1953519}, the coefficients of $\lambda$, $\delta_{\rm irr}$, and $\delta_{i:0}$ in the class of $\overline{\mathfrak{S}}^{r,t}_{g,d}$ coincide with the coefficients of $\lambda$, $\delta_{\rm irr}$, and $\delta_{g-i}$ in the class of the closure of the divisor $(\mathfrak{S}^{r,t}_{g,d})^{ n}$, hence the statement for $c_\lambda$, $c_{\rm irr}$, and $c_{i:0}$ follows from Lemma \ref{PBN}.

The idea to compute the coefficient $c_\psi$ is similar.
Let $({\mathfrak{S}}^{r,t}_{g,d})^{n-1}$ be the divisor in $\M_{g,2}$ obtained by letting all marked points but one in a general element in $\mathfrak{S}^{r,t}_{g,d}$ come together, that is,
\[
({\mathfrak{S}}^{r,t}_{g,d})^{n-1} := \left\{[C,x,y] : \exists \,\ell \in G^r_{d}(C) \,\,\mbox{with}\,\, h^0\left( C,\ell\otimes\mathcal{O}_C\left( {}-x-(n-1)y \right) \right) \geq r+1-t  \right\}.
\]
The class of the closure of $({\mathfrak{S}}^{r,t}_{g,d})^{n-1}$ can be written as 
\[
\left[(\overline{\mathfrak{S}}^{r,t}_{g,d})^{n-1}\right] = \frac{n_{g,r,d,a}}{g(g^2-1)}\left(h_\lambda \lambda+ h_x \psi_x + h_y \psi_y -h_{0:\{x,y\}} \delta_{0:\{x,y\}}-\cdots \right)\in {\rm Pic}(\Mbar_{g,2}).
\]
Note that $h_\lambda=c_\lambda$ and $h_x = c_{\psi}$. Moreover, we have $h_{0:\{x,y\}}= \binom{g+1}{2}$. Indeed, if we let the two marked points in a general element in $(\mathfrak{S}^{r,t}_{g,d})^{n-1}$ collide together, we recover the locus $\Mbar^r_{g,d}(a)$ as above, and $h_{0:\{x,y\}}$ coincides with the coefficient of $\psi$ in the class of $\Mbar^r_{g,d}(a)$ from Lemma \ref{PBN}. Let $C_x:=\{[C,x,y]\}_{x\in C}$ and $C_y:=\{[C,x,y]\}_{y\in C}$ be the curves in $\Mbar_{g,2}$ obtained by fixing one general marked point  and varying the other marked point on a general curve $C$ of genus $g$.
We have
\begin{eqnarray*}
\left[(\overline{\mathfrak{S}}^{r,t}_{g,d})^{n-1}\right]  \cdot C_x &=& \frac{n_{g,r,d,a}}{g(g^2-1)} \left( (2g-1)h_x+h_y-h_{0:\{x,y\}} \right)= T_{g,d}^{r,t}(1),\\
\left[(\overline{\mathfrak{S}}^{r,t}_{g,d})^{n-1}\right]  \cdot C_y &=& \frac{n_{g,r,d,a}}{g(g^2-1)} \left( h_x+(2g-1)h_y-h_{0:\{x,y\}} \right)= T_{g,d}^{r,t}(n-1),
\end{eqnarray*}
where $T_{g,d}^{r,t}(\cdot)$ is as in Proposition \ref{Prop}, whence we recover the coefficient $h_x \equiv c_\psi$.

Intersecting $\overline{\mathfrak{S}}^{r,t}_{g,d}$ with the curve in $\Mbar_{g,n}$ obtained by fixing $n-1$ general marked points on a general curve of genus $g$ and letting one additional marked point vary, we obtain
\[
\frac{n_{g,r,d,a}}{g(g^2-1)} \left( (2g+2n-4)c_\psi -(n-1) b_{0:2}\right) = T_{g,d}^{r,t}(1),
\]
whence we deduce $c_{0:2}$. The coefficients $c_{0:j}$ for $j\geq 3$ are computed recursively. Consider a general $(n-j+1)$-pointed curve of genus $g$, and identify one of the marked points with a moving point on a rational curve having $j$ fixed marked points. The intersection of this test surface with $\overline{\mathfrak{S}}^{r,t}_{g,d}$ is empty, and we have the following relation
\[
j\cdot c_\psi + (j-2) \cdot c_{0:j} - j\cdot c_{0:j-1} =0
\]
whence we deduce $c_{0:j} = (j(j-1)/2) \cdot c_{0:2} - j(j-2) \cdot c_\psi$.
\end{proof}

\section{Divisors of secant planes in symmetric products of a ge\-ne\-ral curve}
\label{symprod}

We close by proving Theorem \ref{Secresult} after Theorem \ref{Secclass}.

\begin{proof}[Proof of Theorem \ref{Secresult}]
Let $\overline{\mathcal{C}}_{g,n}:=\Mbar_{g,n}/\mathcal{S}_n$ be the universal degree-$n$ symmetric product, and let $\pi\colon\Mbar_{g,n}\rightarrow \overline{\mathcal{C}}_{g,n}$ be the quotient map.
Let $u\colon C_n \dashrightarrow \overline{\mathcal{C}}_{g,n}$ be the rational map $u(x_1+\cdots +x_n) = [C,x_1+\cdots +x_n]$, well-defined outside the codimension-$2$ locus of effective divisors of $C$ with support of length at most $n - 2$. 
If $\widetilde{\mathfrak{S}}^{r,t}_{g,d}$ is the effective divisor in $\overline{\mathcal{C}}_{g,n}$ such that $\pi^*\widetilde{\mathfrak{S}}^{r,t}_{g,d} = \overline{\mathfrak{S}}^{r,t}_{g,d}$,
then $u^*(\widetilde{\mathfrak{S}}^{r,t}_{g,d})$ coincides with the divisor $\mathfrak{S}^{r,t}_{g,d}(C)$ in $C_n$. 
Let $\tilde{\delta}_{0:2}$ be the divisor class on $\overline{\mathcal{C}}_{g,n}$ whose pull-back via $\pi$ is ${\delta}_{0:2}$, and let $\tilde{\psi}$ be the first Chern class of the line bundle $\mathbb{L}$ on $\overline{\mathcal{C}}_{g,n}$ defined as $\mathbb{L}[C,x_1+\cdots +x_n]= T^\vee_{x_1}(C)\otimes \cdots \otimes  T^\vee_{x_n}(C)$ over a point $[C,x_1+\cdots +x_n]$ in  $\overline{\mathcal{C}}_{g,n}$. Note that $\pi^*(\tilde{\psi})=\sum_i \psi_i -\sum_{j=2}^n j \cdot\delta_{0:j}$ (\cite{MR3102167}).
One has $u^*(\tilde{\delta}_{0:2})=\delta_C = -\theta +(g+n-1)x$ and $u^*(\tilde{\psi})=\theta + \delta_C +(g-n-1)x = (2g-2)x$ (\cite[Proposition 2.7]{MR1930983}). We deduce
\begin{eqnarray*}
\left[ \mathfrak{S}^{r,t}_{g,d}(C)\right] &=& \frac{n_{g,r,d,a}}{g(g^2-1)} \left((c_{0:2} -2c_\psi)\cdot \theta +((2g-2)c_\psi - (g+n-1)(c_{0:2} -2c_\psi)) \cdot x \right) \\
&=& \frac{n_{g,r,d,a}}{g(n^2-1)} \left(\theta - \frac{g}{n}x \right)
\end{eqnarray*}
where $n_{g,r,d,a}$ is as in Proposition \ref{nn}, and $c_{0:2}$ and $c_\psi$ are as in Theorem \ref{Secclass}.
\end{proof}

\bibliographystyle{alpha}
\bibliography{Biblio.bib}

\end{document}